\documentclass{article}

\usepackage{amsmath,amsthm,amsfonts,amssymb}
\usepackage{graphicx}
\usepackage{tikz-cd}

\newtheorem{lm}{Lemma}
\newtheorem{prop}[lm]{Proposition}
\newtheorem{theorem}[lm]{Theorem}

\newtheorem{cy}[lm]{Corollary}

\theoremstyle{definition}
\newtheorem{df}[lm]{Definition}

\newtheorem{rk}[lm]{Remark}

\usepackage{color,cancel}

\newcommand{\beq}{\begin{equation}}
\newcommand{\eeq}{\end{equation}}
\newcommand{\be}{\begin{enumerate}}
\newcommand{\ee}{\end{enumerate}}
\newcommand{\bp}{\begin{proof}}
\newcommand{\ep}{\end{proof}}
\newcommand{\bi}{\begin{itemize}}
\newcommand{\ei}{\end{itemize}}
\newcommand{\bea}{\begin{eqnarray*}}
\newcommand{\eea}{\end{eqnarray*}}
\newcommand{\bml}{\begin{multline*}}
\newcommand{\eml}{\end{multline*}}

\newcommand{\prn}[1]{\left( #1 \right)}
\newcommand{\bkt}[1]{\left[ #1 \right]}
\newcommand{\set}[1]{\left\{ #1 \right\}}
\newcommand{\abs}[1]{\left| #1 \right|}
\newcommand{\norm}[1]{\left|\left| #1 \right|\right|}

\newcommand{\gen}[1]{\langle #1 \rangle}
\newcommand{\Z}{\mathbb{Z}}
\newcommand{\C}{\mathbb{C}}
\newcommand{\Q}{\mathbb{Q}}

\edef\storedcatcodeat{\the\catcode`\@} \catcode`\@=11

\catcode`\@=\storedcatcodeat\relax

\begin{document}

\title{Quantifying separability in limit groups via representations}

\author{Keino Brown, Olga Kharlampovich}

\maketitle
\begin{abstract}
\noindent  We show that for any finitely generated subgroup $H$ of a limit group $L$ there exists a finite-index subgroup $K$ containing $H$, such that $K$ is a subgroup of a group obtained from $H$  by a series of extensions of centralizers and free products with $\mathbb Z$. If $H$ is non-abelian, the $K$  is fully residually $H$.  We also show that for any finitely generated
subgroup of a limit group, there is a finite-dimensional representation of the limit group which separates the
subgroup in the induced Zariski topology. As a corollary, we establish a polynomial upper bound
on the size of the quotients used to separate a finitely generated subgroup in a limit group. This generalizes the results in \cite{L}.  Another corollary is that a hyperbolic limit group satisfies the Geometric Hanna Neumann conjecture.      
\end{abstract}
\section{Introduction}

\noindent A group is said to \textit{retract} onto a subgroup if the inclusion map of the subgroup into the group admits a left-inverse. In which case, the left-inverse is called a \textit{retraction} and the subgroup a \textit{retract}. In \cite{Wilton}, Wilton proves that if $H$ is a finitely generated subgroup of a limit group $L,$  $g\in L-H$, then $H$  is a retract of some finite-index subgroup $K\leq L$ which contains $H$ but not $g$ \cite{Wilton}. We will refer to the smallest set of groups containing all finitely generated free groups that is closed under extensions of centralizers as ICE. By \cite{KM}, limit groups are precisely the finitely generated subgroups of groups from ICE. We will modify the construction, from \cite{Wilton}, of a finite-index subgroup $K\leq G$,  where $G$ is an ICE group, in such a way that not only is there a retraction $K\rightarrow H$, but, for a non-abelian $H$, a discriminating family of retractions (for each finite set $S$ of non-trivial elements in $K,$  there is a retraction from $K$ onto  $H$ that is injective on $S$). In other words, $K$ is fully residually $H$. This finite-index subgroup $K$ will be a  group obtained from $H$ by a finite chain of groups $H=K_0<\ldots <K_n=K$, where $K_{i+1}$ is either $K_i\ast F$, where $F$ is some free group  or $K_{i+1}$ is an extension of a centralizer in $K_i$. We will call a group obtained by such a chain an $H-$GICE group. (If $H$ is free, then the classes of $H$-GICE groups and $ICE$ groups containing $H$, coincide.) It is well known that an extension of a centralizer of a limit group $G$ is fully residually $G$. This was first proved in \cite{Lyndon} (one can find a detailed proof, for example, in  \cite[Lemma 3.7]{KMS}.) It is also known that a free product of a non-abelian limit group $G$ and  a free group is fully residually $G$. Therefore, each group in the chain used to construct $K$ is fully residually $H$.
\begin{theorem}\label {th3} Let $G$ be an ICE-group, $H$ be a finitely generated  subgroup, and $g\in L-H$. Then, there exists a finite-index subgroup $K$ of $G$ such that $H\leq K$, $K$ is an $H$-GICE group,  and $g\not\in K$.
\end{theorem}
\begin{cy}\label {th6} Let $L$ be a limit group, $H$ be a finitely generated  subgroup, and $g\in L-H$. Then, there exists a finite-index subgroup $K$ of $L$ such that $H\leq K$, $K$ is a subgroup of an $H$-GICE group,  and $g\not\in K$.
\end{cy}

\noindent This theorem  implies the following result.
\begin{theorem}\label {th2} Let $L$ be a limit group, $H$ be a finitely generated non-abelian subgroup, and $g\in L-H$. Then, there exists a finite-index subgroup $K$ of $L$ such that $H\leq K$, $K$ is fully residually $H,$ and $g\not\in K$.
\end{theorem}

\noindent Theorem \ref{th2} is also true   when $L$ is abelian and, therefore, H is abelian (see  Remarks at the end of the proof).     In the case when $H$ is abelian and  $L$ is non-abelian  a finite-index subgroup of $L$ cannot be fully residually $H$.


\begin{theorem} \label{th1} Let $L$ be a limit group. If $H$ is a finitely generated non-abelian subgroup
of $L,$ then there is a  faithful representation $\rho _H:L\rightarrow GL(V)$ such that $\overline {\rho _H(H)} \cap\rho_H(L)=\rho _H(H)$, where $\overline {\rho _H(H)} $ is the Zariski closure of  $\rho _H(H).$ 
\end{theorem}

\noindent Likewise, this theorem is true when $L$ is abelian.

\begin{cy} \label{co1} Let $L$ be a limit group and $S$ be a finite generating set for $L.$ If $H\leq L$ is a finitely generated subgroup, then there exists a constant $N>0$ such that for each $g\in L- H,$ there exist a finite group $Q$ and a homomorphism $\varphi:L\longrightarrow Q$ such that $\varphi\prn{g}\notin\varphi\prn{H}$ and $\abs{Q}\leq\norm{g}_S^N.$ If $K=H\ker\varphi,$ then $K$ is a finite-index subgroup of $L$ whose index is at most $\abs{Q}\leq\norm{g}_S^N$ with $H\leq K$ and $g\notin K.$ Moreover, the index of the normal core of the subgroup $K$ is bounded above by $\abs{Q}$.
\end{cy}

 \noindent To use Theorem \ref{th1} for the proof of this corollary in the case when $L$ is non-abelian and $H$ is abelian we can take instead of $H$ a non-abelian subgroup $H_1=H*\langle x\rangle$ for a suitable  element $x$. 

Our Theorem \ref{th1} and Corollary \ref{co1} generalize results for free and surface groups from \cite{L}. We use \cite{L} to deduce Corollary \ref{co1} from Theorem \ref{th1}.
Corollary \ref{co1} establishes polynomial bounds on the
size of the normal core of the finite index subgroup used in separating $g$ from $H$. The constant $N$ explicitly depends on the subgroup $H$ and the dimension of $V$ in
Theorem \ref{th1}. For a general finite index subgroup, the  upper bound for the index
of the normal core is factorial in the index of the subgroup. It is for this reason that
we include the statement about the normal core of $K$ at the end of the corollary.

Recently, several effective separability results have been established; see \cite{BR1}-\cite{L},  \cite{K1}-\cite{K3},  \cite{P1}-\cite{S}.
Most relevant here are  papers \cite{L}, \cite{HP}.  
The methods used in \cite{HP} give linear bounds in terms of the word length of $|g|$ on the
index of the subgroup used in the separation but do not  produce polynomial
bounds for the normal core of that finite index subgroup. We can also obtain bounds on the index of the separating subgroup on
the order of magnitude $C|g|,$ where $C$ is a constant depending on $L$ and $H$.

In Section \ref{HN} we will formulate the Geometric Hanna Neumann conjecture by Antolin and Jaikin-Zapirain for limit groups and give a proof (due to Jaikin-Zapirain)  that Theorem \ref{th3}  implies the conjecture for hyperbolic limit groups (Theorem \ref{thHN}).

\vspace{2mm}
\section{Preliminaries}

\begin{df} A family $\mathcal F$ of $H$-homomorphisms (identical on $H$) from a group $G$ onto a subgroup $H$ is called a \textit{discriminating family} if for any finite set $S$ of non-trivial elements in $G$ there exists a homomorphism $\psi\in\mathcal F$ such that for any $g\in S$, $\psi (g)\neq 1.$ We say $G$ is  \textit{fully residually} $H$ if there exists a discriminating family of $H$-homomorphisms from $G$ to $H$.
\end{df}
\begin{df}
    Let $G$ be a group and $C_G(u)$ denote the centralizer of an element $u\in G$. An \textit{extension of a centralizer} of $G$ is the group $$(G,u)=\gen{G,t_1,\ldots ,t_k\mid\bkt{c,t_i}, c\in C_G(u), [t_i,t_j], i,j=1,\ldots ,k}.$$  Similarly, if we extend centralizers of several non-conjugated elements $u_1,\ldots ,u_m$ in $G$ we denote the obtained group by
  $(G,u_1,\ldots ,u_m).$  
    
     An \textit{iterated extension of centralizers}  is obtained by finitely many applications of this construction to a finitely generated free group and is called an ICE-group. In this case we can  assume that each centralizer is extended only once. In other words, on each step  $C_G(u)$ is cyclic.

  Let \begin{equation} \label{chain1} F=G_0<G_1=(G,u_1)<\ldots < G_n=(G_{n-1},u_n)=G\end{equation} be a chain of centralizer extensions to obtain an ICE-group $G$. Then we always assume that in this chain centralizers in $G_i$ are extended before centralizers in $G_{i+1}.$ 
 We can modify this chain the following way
\begin{equation} \label{chain2} F=G_0<G_{i_1}<\ldots <G_{i_k}=G, \end{equation}  
 where $G_{i_1}=(G_0, u_1,\dots ,u_{i_1})$, where $u_1,\ldots ,u_{i_1}$ are in $G_0$  is obtained from $G_0$ by extending all the centralizers of elements from $G_0$ that appear in the first chain. Similarly $G_{i_{j+1}}$ is obtained from $G_{i_j}$ by extending all the centralizers of elements in $G_{i_j}$ that were extended in the first chain.     
\end{df}

\begin{df}
    Let $G$ be an ICE group. Then associated with $G$ is a finite $K(G,1)$ space, called an \textit{\textit{ICE space}}, which is constructed as follows:
    \begin{enumerate}
        \item If $G$ is free, then take $X=K(G,1)$ to be a  compact graph of suitable rank.

        \item If $G$
is obtained from a group $G'$ by an extension of  a centralizer and $Y=K(G',1),$ then given an essential closed curve $\partial_{+}:S^{1}\longrightarrow Y$ representing a generator of  $C_{G'}(g')$  and a coordinate circle $\partial_{-}:S^{1}\longrightarrow T,$ where $T$ is a torus, take   
$$X=Y\sqcup\prn{\bkt{-1,1}\times S^1}\sqcup T,$$ 
identifying $\prn{
\pm 1, \theta}\in \bkt{-1,1}\times S^1$ with $\partial_{\pm}\prn{\theta}.$
\end{enumerate}
\end{df}

\begin{rk} Associated to each ICE space $X$ is a graph of spaces decomposition whose vertices are $Y$ and $T$ and edges are circles.
\end{rk}

\begin{df} \label{df8} A group $G$ is an \textit{$H$-GICE group} if  it is obtained from $H$ by a series of free products with free groups  and extensions of centralizers. Here, GICE stands for  generalized iterated centralizer extension.
\end{df} 

 \noindent If $H$ is a non-abelian limit group, then any $H$-GICE group and its subgroups containing $H$ are fully residually $H$, see, for example \cite{KMS}. Therefore, Theorem \ref{th3} implies Theorem \ref{th2}.
\begin{df}Let each of the following spaces have a chosen basepoint, and suppose that the
maps are basepoint preserving. Let $\rho :(B',b')\rightarrow (B,b)$ be a covering map. Let $\delta :(A,a)\rightarrow (B,b)$ be a map, where $A$ is a connected complex (in our case $A$ will be a loop). Let $\kappa :(A',a')\rightarrow (A,a)$ be the smallest cover of $(A,a)$ such that the map $\delta\circ\kappa$ has a lift $\delta '$. We call $\delta ':(A',a')\rightarrow (B',b')$ the \textit{elevation} of $\delta$.

Two elevations ${\delta _1}': {A_1}'\rightarrow {B}'$ and ${\delta _2}': {A_2}'\rightarrow {B}'$  are {\em isomorphic} if there exists a homeomorphism $\iota : {A_1}'\rightarrow {A_2}'$ covering the identity map on $A$, such that $
{\delta _1}'={\delta _2}'\circ\iota.$

\end{df}
For more information on elevations we refer to \cite[Section 2]{Wilton}.


\begin{df}
    Let $X$ and $X'$ be graphs of spaces ($X'$ is not assumed to be connected). A \textit{pre-covering }is a locally injective map $X'\longrightarrow X$ that maps vertex spaces and edge spaces of $X'$ to vertex spaces and edge spaces of $X$ respectively and restricts to a covering on each vertex space and each edge space. Furthermore, for each edge space $e'$ of $X'$ mapping to an edge space $e$ of $X,$ the diagram of edge maps 
       \[\begin{tikzcd}
e' \arrow{r}{\Bar{\partial}_{\pm}} \arrow[swap]{d}& \Bar{V}_{\pm} \arrow{d} \\
e \arrow{r}{\partial_{\pm}} & V_{\pm}
\end{tikzcd}
\]

    \noindent is required to commute. The domain $X'$ is called a \textit{pre-cover}. 
\end{df}

\begin{df}(\cite[Definition 3.1]{Wilton}) Let $X$ be a complex, $ X'\longrightarrow X$ be a covering, and $$\mathcal{L}=\{\delta_i:C_i\longrightarrow X\}$$ be a finite collection of independent, essential loops. The cover $X'$ is said to be \textit{tame over} $\mathcal{L}$ if the following holds: let  $\Delta\subset X'$ be a finite subcomplex
 and $$\mathcal{L}'=\{\delta_j':C'_j\longrightarrow X'\}$$ be a finite collection of pairwise non-isomorphic infinite degree elevations, each of which is an  elevation of some loop in $\mathcal{L}.$ Then for all sufficiently large positive integers $d$ there exists an intermediate finite-sheeted covering
 $$X'\longrightarrow\hat{X}\longrightarrow X$$ such that

 \begin{enumerate}
     \item each $\delta_j'$ descends to some degree $d$ elevation $\hat{\delta_j}$ 
     
     \item the $\hat{\delta_j}$
 are pairwise non-isomorphic,     \item $\Delta$ embeds into $\hat{X},$ and
     \item there exists a retraction $\rho: \pi_1(\hat{X})\longrightarrow\pi_1(X')$ such that $$\rho(\hat{\delta}_{j*}(\pi_1(\hat{C_j})))\subset\delta_{j*}'(\pi_1(C_j'))$$ for each $j.$
 \end{enumerate}
\textit{Remark.} We will also say a covering $X'\longrightarrow X$ is tame over a given set of finite independent, essential loops whenever its domain $X'$ is.
\end{df}
Notice, that covers of tori are tame over coordinate circles, see \cite[Lemma 3.3]{Wilton}.

\begin{df} The cover $X'$ is {\em strongly tame over} $\mathcal L$ if it is tame over $\mathcal L$ and   $\pi_1(\hat X)$ is a  $(\pi_1(X')\ast F)$-GICE group, where $F$ is a free group with  basis $\{\hat\delta_{j\ast} (\pi_1(\hat C_j))\}.$ 
\end{df}

\begin{df}
     A group $G$ is said to admit a \textit{local GICE structure} if for each finitely generated  subgroup $H\leq G$ and a finite set of elements $g_i\not \in H$ one can construct a finite-index subgroup $K$ containing $H$ and not containing these elements such that $K$ is a $H$-GICE group.
     
 \end{df}


\section{Proof of Theorems \ref{th3}, \ref{th2}}

\noindent We will  follow the construction in \cite{Wilton} changing it a couple of times to prove a theorem similar to  \cite[Theorem 3.8]{Wilton}.   One difference is that we will use induction on the number of steps in chain (\ref{chain2}) while \cite[Theorem 3.8]{Wilton} is proved by induction on the number of steps in chain (\ref{chain1}).

\vspace{2mm}\noindent Let $X$ be an ICE space constructed by gluing several tori $T_1,\ldots ,T_k$ to a simpler ICE space $Y$ with edge spaces being loops. Let $H\subset\pi_1(X)$ be a finitely generated subgroup and $X^H\rightarrow X$ be the corresponding covering. Then $X^H$ inherits a graph of spaces decomposition, with vertex spaces the connected components of the pre-images of the vertex spaces of $X$ and edge spaces and maps given by all the (isomorphism classes of) elevations of the edge maps to the vertex spaces of $X^H$. Let  $X'\subseteq X^H$ be a core of $X^H$. A {\bf core} is a connected sub-graph of spaces with finite underlying graph such that the inclusion map is a $\pi _1$-isomorphism. Since $H$ is finitely generated, a core exists. Let $\Delta\subset X^H$ be a finite subcomplex. Enlarging $X'$ if necessary we can assume $\Delta\subset X'$.

 \vspace{2mm}\noindent Replacing the tameness hypothesis in  \cite[Proposition 3.4]{Wilton}  by strong tameness, we have the following. 


\begin{prop}\label{prop3.4} (\textit{Passing to finite-sheeted pre-covers}) Let $X$ be an ICE space constructed by gluing several tori $T_1,\ldots ,T_k$ to a simpler ICE space $Y$ with edge spaces being loops,  
Let $X'\rightarrow X$ be a pre-covering with finite underlying graph. Every vertex space $V'$ of $X'$ covers some vertex space $V$ of $X.$ Assume that each $Y'$  is strongly tame over the set of edge maps incident at $Y$.  Let $\Delta\subset X'$ be a finite subcomplex.  Then  there is  a finite-sheeted intermediate pre-covering 

$$X'\rightarrow \Bar{X}\rightarrow X$$ such that 

\be

\item $\Delta$ embeds into $\Bar{X};$ and

\item
$\pi _1(\bar X)$ is  a $\pi _1(X')$-GICE group. 
\ee

\end{prop}

\begin{proof} Let $\Delta_0$ be a finite complex that contains $\Delta$ and all the compact edge spaces of $X'.$ Let $V'$ be a vertex space of $X'$ covering the vertex $V$ of $X.$ Set $\Delta_{V'}=V'\cap\Delta_0$ and consider the edge maps $\partial_{i}':e_i\rightarrow V'$ of edges $e_i$ incident at $V'$ that are infinite-degree elevations of $\partial_{\pm}:e\rightarrow V.$ Since each $V'$ is tame over  the set of edge maps incident at $V$ and each $Y'$ is strongly tame over  the set of edge maps incident at $Y$,  for all sufficiently large $d$ there exists a  an intermediate finite-sheeted covering $$V'\rightarrow \Bar{V}\rightarrow V$$ such that

\begin{enumerate}
    
     \item $\Delta_{V'}$ embeds into $\Bar{V},$

     \item each $\partial_i'$ descends to some degree $d$ elevation $\Bar{\partial_i}$ of $\partial_{\pm}.$  \end{enumerate}

 \noindent If $d$ is large enough, we can take it to be the same $d$ over all vertex spaces of $X'.$ Let $\Bar{X}$ be the graph of spaces with the same underlying graph as $X',$ but with the corresponding $\Bar{V}$ in place of $V'.$ If $e'$ is an edge space of $X'$ then the edge map
$$\partial_{\pm}:e'\rightarrow V'$$ descends to a finite-degree map $\partial_{\pm}: \Bar{e}_{\pm}\rightarrow \Bar{V}.$ Because $\Bar{e}_+\rightarrow e$ and $\Bar{e}_{-}\rightarrow e$ are coverings of $e$ with the same degree, we have a finite-sheeted pre-cover $\Bar{X}.$  
By construction, $\Delta$ embeds into $\Bar{X}.$   Since the compact edge spaces are added to $\Delta$, non-isomorphic finite degree elevations are mapped into non-isomorphic elevations. This implies that $\pi _1(\bar X)$ decomposes as a graph of groups, with the same underlying graph as the decomposition of $\pi _1(X')$. 

Consider a non-abelian vertex group $\pi_1(\bar V)$   of $\pi _1(\bar X)$ (this means $V=Y$). To obtain $\pi _1(\bar V)$ we first take  $\pi_1(V'')=\pi_1(V')\ast F$,  the free product  with cyclic groups corresponding to elevations of degree $d$ obtained from infinite degree elevations of edge maps, and then by a series of extensions of centralizers and free products with free  groups.  A cyclic fundamental group of an elevation of degree $d$ obtained from an infinite degree elevation of an edge map extends the abelian fundamental group of an infinite  cover of some torus $T_i$. On the group level this corresponds to the extension of    the centralizer of an abelian free factor of $\pi _1(X')$ (and, therefore, extension of  a centralizer of $\pi _1(X')$ itself, because the extending element is in the free factor $F$ of $\pi_1(V')\ast F$. So, to obtain  $\pi _1(\bar X)$ we first extend centralizers of $\pi _1(X')$ corresponding to abelian free factors. We also extend centralizers of all $\pi_1(T')$, where $T'$ covers some $T_i$ so that all $\bar T$'s become finite covers. Denote by $X''$ the pre-cover that is obtained from $X'$ by replacing the covers of tori by  finite covers as above and  replacing $V'$ by $V''$ for each $V$ that is not a torus.
Second, we notice that the free constructions  that were applied to each $\pi _1(V')\ast F$ to obtain $\pi _1(\bar V)$, for $V'$ covering the vertex $V=Y$, can be thought as applied to the whole group $\pi _1(X'').$  Replacing each $V''$ by $\bar V$ for covers $V'$ of the non-abelian vertex group $V=Y$ we obtain $\bar X$.\end{proof}

\begin{lm}(\cite[Lemma 3.5]{Wilton}
    Let $T$ be a torus and $\delta:S^1\rightarrow T$ be an essential loop. Then for every positive integer $d$ there exists a finite-sheeted covering $\hat{T}_d\rightarrow T$ so that $\delta$ has a single elevation $\hat{\delta}$ to $\hat{T}_d$ and $\hat{\delta}$ is of degree $d.$ 
\end{lm}

\begin{lm}(cf \cite[Lemma 3.6]{Wilton})\label{lm3.6}
    Let $Y$ be a space such that $\pi_1(Y)$ has local GICE structure and $\delta:S^1\rightarrow Y$ be a based essential loop. Then for every positive integer $d$ there exists a finite-sheeted covering $\hat{Y}_d\rightarrow Y$ so that $\delta$ has an elevation $\hat{\delta}$ of degree $d$ to $\hat{Y}_d$ and  $\pi_1(\hat{Y}_d)$  is an $\gen{\hat{\delta}}$-GICE group. 
\end{lm}

\begin{proof} Because $\pi_1(Y)$ has local GICE structure, 
for every positive integer $d$ there exists a finite-sheeted covering $\hat{Y}_d\rightarrow Y$ so that $\hat{Y}_d\rightarrow Y$ is a $\gen{{\delta ^d}}$-GICE group. Note that $\delta ^k\not\in \pi_1(\hat{Y}_d)$ for $0<k<d$, therefore $\hat\delta$  is an elevation of degree $d$.\end{proof}


\begin{prop} (cf \cite[Proposition 3.7]{Wilton})\label{prop3.7}\textit{(Completing a finite-sheeted pre-cover to a cover)} Let $X$ be an ICE space constructed by gluing together tori $T_1,\ldots ,T_k$ and a simpler ICE space $Y,$  as above. Assume that $\pi_1(Y)$ admits a local GICE structure. Let $\Bar{X}\rightarrow X$ be a finite-sheeted connected pre-covering. Then there exists an inclusion $\Bar{X}\hookrightarrow \hat{X}$ extending $\Bar{X}\rightarrow X$  to a covering $\hat{X}\rightarrow X$ such that $\pi_1(\hat{X})$ is a  $\pi_1(\Bar{X})-GICE$ group.
    
\end{prop}

\begin{proof} Follows the proof of \cite[Proposition 3.7]{Wilton}. The addition of copies of $T_{i,d}$ correspond to extensions of centralizers. The addition of $Y_d$'s correspond by Lemma \ref{lm3.6} to taking a free product with infinite cyclic group and then a GICE over the obtained  group. Indeed, $\pi _1(Y)$ has a local GICE-structure, therefore $Y_d$ is 
$C-GICE$  group, where $C$ is a cyclic group generated by the  boundary element. \end{proof}

A collection of elements $g_1,\ldots ,g_n$ of a group $G$ is called {\em independent} if whenever there exists $h\in G$ such that $g_i^h$ and $g_j$ commute, then, in fact, $i=j.$
\begin{prop}(cf \cite[Proposition 3.8]{Wilton}) Let $X$ be an ICE space constructed by gluing together tori $T_1,\ldots ,T_k$ and a simpler ICE space $Y$. Let $H\leq \pi_1(X)$ be a finitely generated subgroup and let $X^H\rightarrow X$ be the corresponding covering. Suppose $\mathcal L$ is a (possible empty) set of hyperbolic loops that generate maximal cyclic subgroups of $\pi _1(X)$. Then $X^H$ is strongly tame over $\mathcal L$. \end{prop}
\begin{proof}
The proof is an induction on the length of the chain (\ref{chain2}).  Notice, that the induction basis holds by \cite[Corollary 1.8]{Wilton}. Indeed, if $H$ is a finitely generated subgroup of $\pi _1(X),$ where $X$ is a graph, then the cover $X^H$ is strongly tame over the set of independent elements $\{\gamma _i\}$ such that each generate a maximal cyclic subgroup, because it is tame and for a finite-sheeted intermediate covering 
$$X^H\rightarrow \hat X\rightarrow X$$ 
 $\pi _1(\hat X)=H\ast F$, where $F$ is a  free group. 
 
 Fix a finitely generated non-abelian subgroup  $H\subset\pi_1(X)$, and let  $X^H\rightarrow X$ be the corresponding covering. There exists a core  $X'\subseteq X^H$. Let $\Delta\subset X^H$ be a finite subcomplex. Enlarging $X'$ if necessary we can assume $\Delta\subset X'$, infinite degree elevations of hyperbolic loops $\{\delta _i\}$ are first restricted to  elevations $\{{\delta _i}'\}$ and then made disparate. This is possible  by  \cite[Lemma 2.24]{Wilton}  without changing the fundamental group.

As in the proof of \cite[Theorem 3.8]{Wilton},  $X'$ is  extended to a pre-cover $\bar X$ where elevations $\{{\delta _i}'\}$ are extended to full elevations $\bar\delta _j:\bar D_j\rightarrow \bar X$ of degree $d$
 by \cite[Lemma 2.23]{Wilton}. By \cite[Lemma 2.23]{Wilton}, $\pi _1(\bar X)=\pi _1(X')\ast F$, where $F$ is a free group generated by $\pi_1(\bar \delta _{i\ast}(\bar D_j))$'s. 
Enlarging $\Delta$ again we assume that the images of the $\bar\delta _j$ are contained in $\Delta$.


\noindent By Proposition \ref{prop3.4} there exists an intermediate finite-sheeted pre-covering 
$$\bar X\rightarrow\hat X\rightarrow X,$$ into which $\Delta$ injects.  Since $\Delta$  injects into $\hat X$ we have that $\bar\delta _j$ descends to an elevation $\hat \delta _j=\bar\delta _j$. 

Finally, $\hat X$ can be extended to a finite sheeted covering $\hat X^+$ by Proposition \ref{prop3.7}. 

We have that $\Delta$ injects into $\hat X^+$. By Proposition \ref{prop3.7}, $\pi _1(\hat X^+)$ is a  $\pi _1(\hat X)$-GICE group. By Proposition \ref{prop3.4}, $\pi _1(\hat X)$ is a  $\pi _1(\bar X)$-GICE group. And $\pi _1(\bar X)=\pi _1(X')\ast F$. Therefore, by transitivity, $\pi _1(\hat X^+)$ is a  $\pi _1(X')\ast F$-GICE group. Since $H=\pi _1(X'),$ the proposition is proved.
\end{proof}

\noindent  Theorem \ref{th3} follows from the proposition  (with the empty set $\mathcal L$). Since every limit group is a subgroup of an ICE-group by \cite{KM2}, Corollary \ref{th6} follows from Theorem \ref{th3}. If $H$ is non-abelian, then
$H$-GICE groups are fully residually $H$ and subgroups of fully residually $H$ groups that contain $H$ are also fully residually $H$. Therefore   Theorem \ref{th2} follows from Theorem \ref{th3}. \vspace{2mm}

\noindent {\bf Example 1}. Let us illustrate the proof of Theorem \ref{th2} with an example when $L$ is just an extension of a centralizer of a free group. Consider the group $$L=F(a,b)\ast_{\langle a\rangle}\langle a,t,|[a,t]=1\rangle, $$ where $F(a,b)$ is a free group, a subgroup $$H=\langle a^2,b^2\rangle \ast _{\langle a^2\rangle}{\langle a^2\rangle}\ast_{\langle a^2\rangle}{\langle a^{2t},b^{2t}\rangle}$$ and $g=\Delta =b \not \in H.$  Let us construct a finite-index subgroup $K$ such that $H\leq K$, $b\not\in K$ and $K$ is an $H$-GICE group. 

In Fig. \ref{Pre} we show the space $X$ such that $L=\pi _1(X)$. Here $X$ is a graph of spaces with one edge and two vertices. The loops labelled by $a$ and $t$ are generating loops of the torus $T$ with a fundamental group $\langle a,t,|[a,t]=1\rangle $ and the bouquet of loops labelled by $a$ and $b$ has a fundamental group $F(a,b)$.  
A pre-cover $X'$ corresponding to $H$ is  a pre-cover with the finite graph. It is a graph of spaces with two edges and three vertices,  $H=\pi _1(X').$
The space corresponding to the vertex in the middle is the cylinder that is an infinite cover of the torus $T$. The other two vertex spaces are infinite covers of the bouquet of loops. 
\begin{figure}[ht!]
\centering
\includegraphics[width=0.5\textwidth]{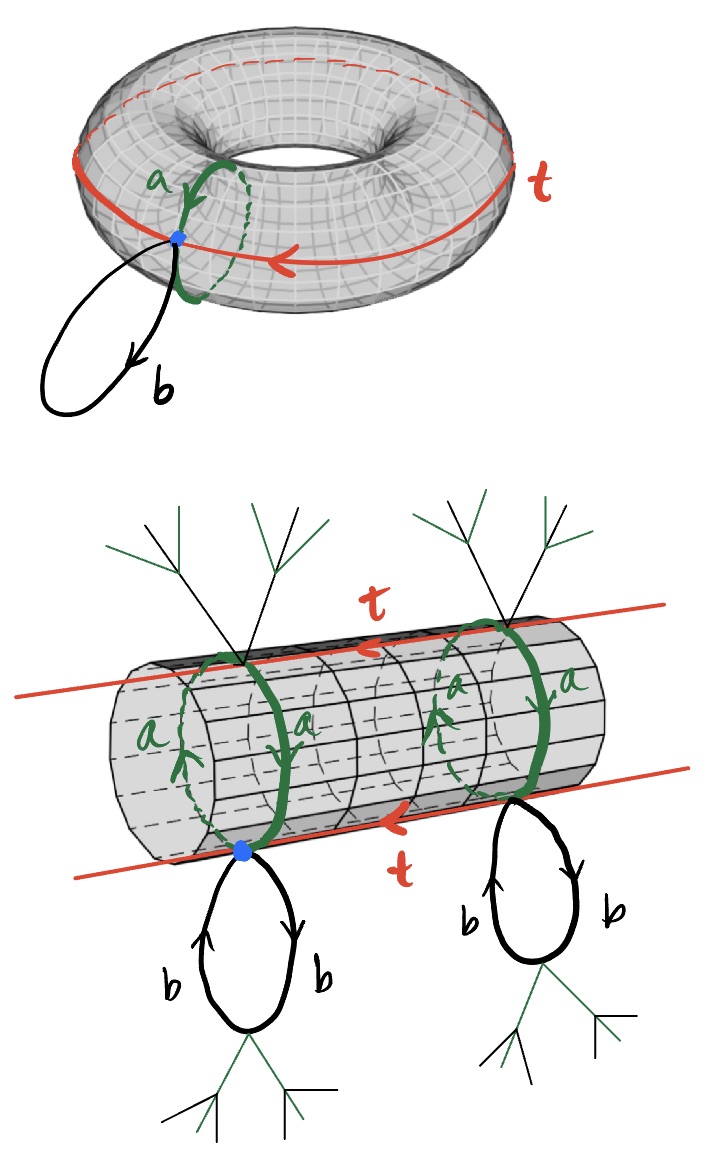}
\caption{ICE space $X$ and a pre-cover $X'$ with a  finite graph}
\label{Pre}
\end{figure}
\newpage
In Fig. \ref{FSPre} we make $X'$ into a finite-sheeted pre-cover $\bar X$ as it is done in Proposition \ref{prop3.4}. The space $\bar X$ has the same underlying graph as $X'$, but the vertex spaces are now finite covers of the vertex spaces of $X$. The torus with the fundamental group generated by $a^2,t^2$ is a cover of $T$ of degree 4. Two other vertex spaces are graphs that are covers of degree 3 of the bouquet of loops in $X$. We have $$\pi _1(\bar X)= \langle a^2,b^2,a^{-1}ba, b^{-1}ab\rangle \ast _{\langle a^2\rangle} \langle a^2,t^2\rangle\ast_{\langle a^2\rangle}\langle a^{2t},b^{2t},t^{-1}a^{-1}bat, t^{-1}b^{-1}abt\rangle.$$  We have that $\pi _1(\bar X)$ is obtained from $H$ by taking a free product with $\langle a^{-1}ba, b^{-1}ab\rangle$ and 
$\langle t^{-1}a^{-1}bat, t^{-1}b^{-1}abt\rangle$ and then extending the centralizer of $a^2$ by $t^2$. There are two hanging elevations of the loop labelled by $a$ in $\bar X$. They both have degree 1. 
\begin{figure}[ht!]
\centering
\includegraphics[width=0.5\textwidth]{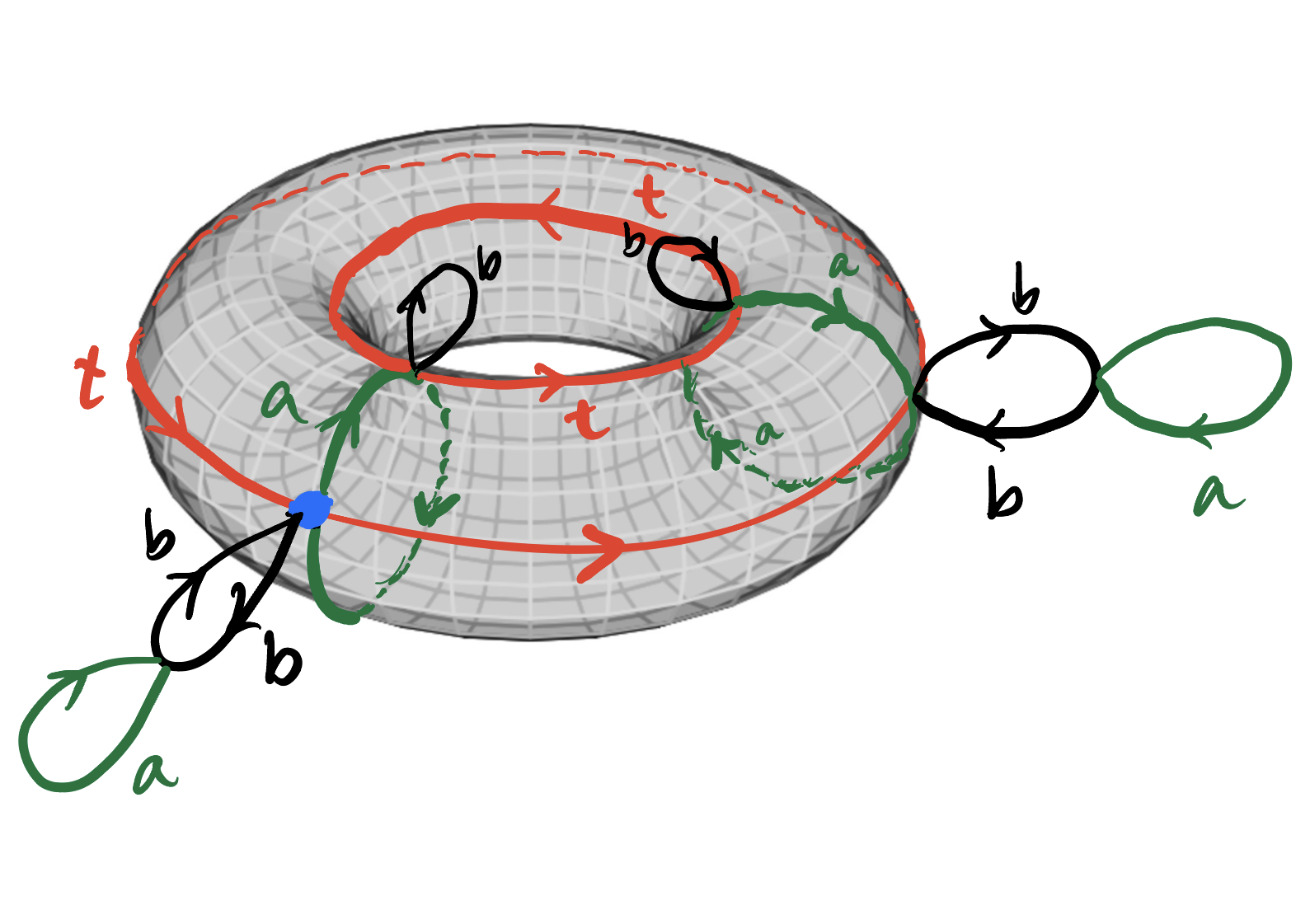}
\caption{Finite-sheeted pre-cover $\bar X$}
\label{FSPre}
\end{figure}
\newpage
\noindent Figure \ref{cover} shows a finite cover $\hat X$ of $X$. It is obtaiinef from $\bar X$ by attaching two tori $T_1$ to the hanging elevations of the loop labelled by $a$ (as in Proposition \ref{prop3.7}). Then $K=\pi _1(\hat X)$ is obtained from $\pi _1(\bar X)$ by extending centralizers of $b^{-1}ab$ (by $b^{-1}tb$) and of $t^{-1}b^{-1}abt$ (by $t^{-1}b^{-1}tbt$). Therefore $K$ is an $H$-GICE group, $b\not \in K$. Notice that $[L:K]=6.$

\begin{figure}[ht!]
\centering
\includegraphics[width=0.6\textwidth]{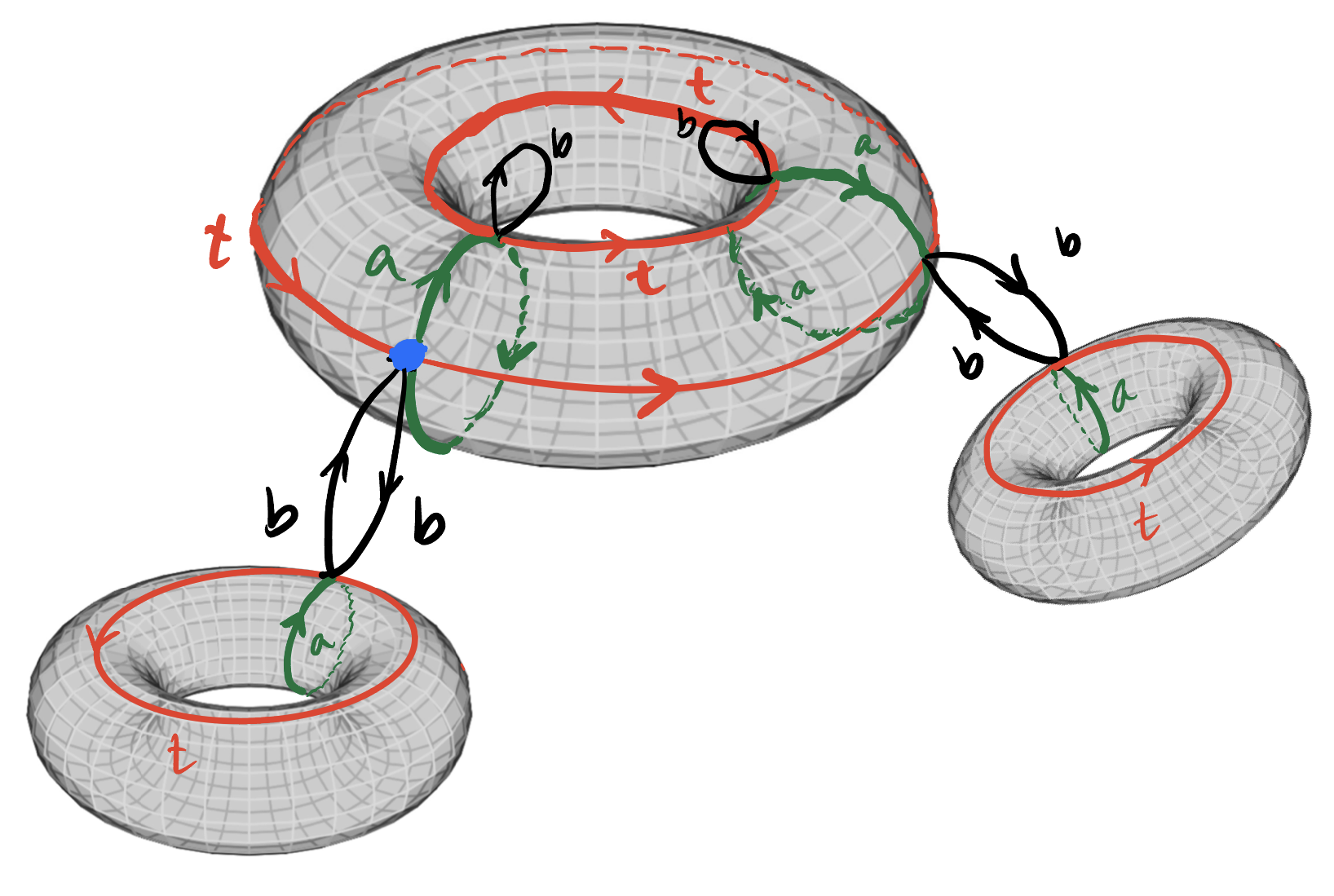}
\caption{Finite cover $\hat X$}
\label{cover}
\end{figure}

{\bf Example 2} (Figures \ref{Pre2}-\ref{cover2}) Now with the same $L$ we take $$H=\langle a^2,b^2\rangle \ast _{\langle a^2\rangle}{\langle a^2\rangle}\ast_{\langle a^2\rangle}{\langle a^{2t},b^{2t}\rangle}\ast \langle t^{ba}\rangle$$ and  $g=\Delta =b \not \in H.$  In this example we will have an edge in $X'$ corresponding to an infinite degree elevation of the loop labelled by $a$. Then
$\pi _1(\bar X)$ is obtained from $H$ by the following chain:
$H< H_1$, where $$H_1=\langle H, a^{ba}, t^2|[a^{ba},t^{ba}]=1, [a^2,t^2]=1\rangle, $$
$H_1<\pi _1(\bar X)$, where
$$\pi _1(\bar X)=H_1\ast\langle a^t, a^{b^{-1}a}, b^{3a},b^{at},a^{bt}\rangle ,$$ where the last group is freely generated by the five given elements. So, to obtain $pi _1(\bar X)$ from $H$ we used two centralizer extensions and then a free product with a free group. To obtain $K=\pi _1(\hat X)$ from $pi _1(\bar X)$ we make three centralizer extensions, as Figure \ref{cover2} shows and obtain s finite cover of $X$ of degree $8$.

\begin{figure}[ht!]
\centering
\includegraphics[width=0.6\textwidth]{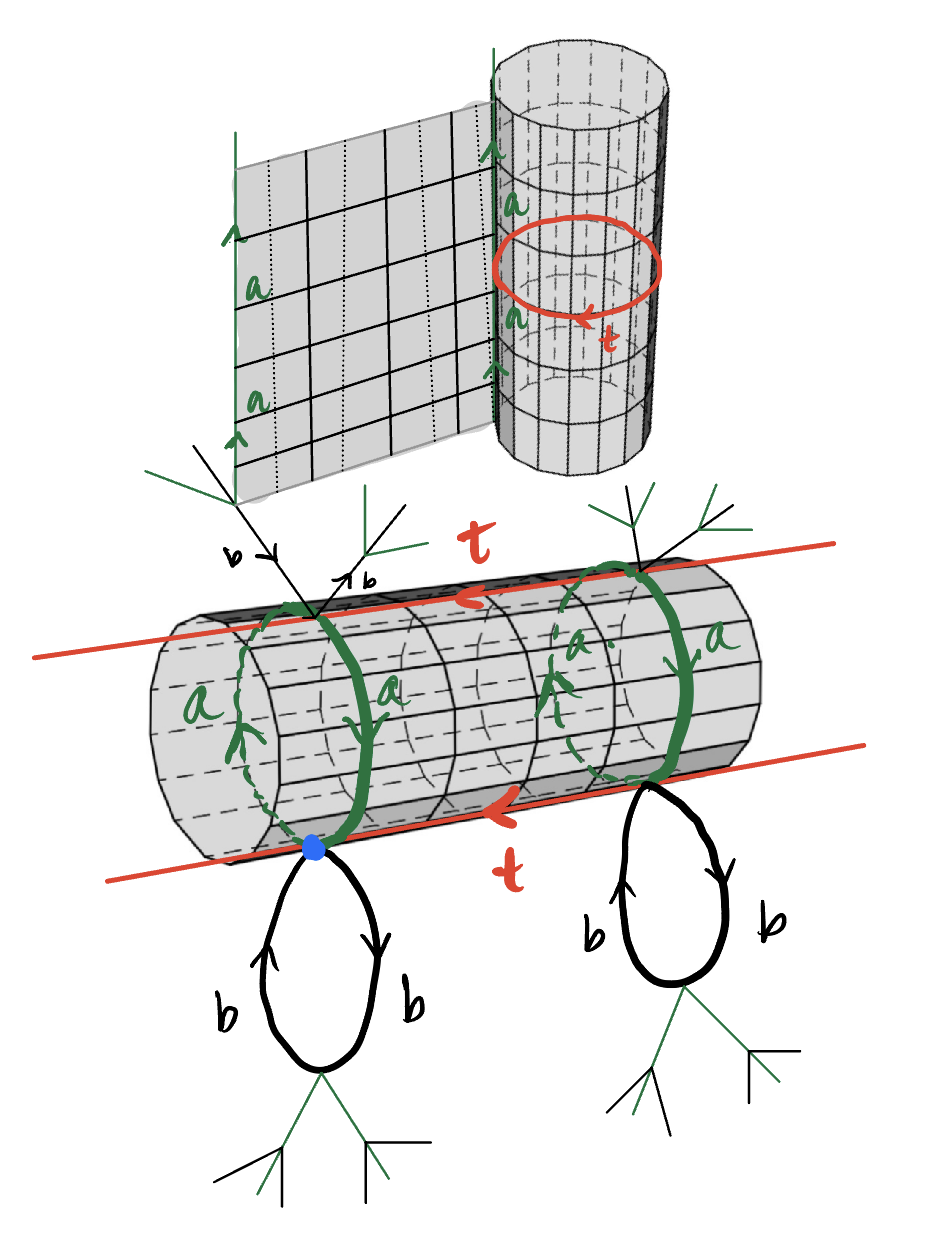}
\caption{Pre-cover $X'$ with finite graph}
\label{Pre2}
\end{figure}
\begin{figure}[ht!]
\centering
\includegraphics[width=0.6\textwidth]{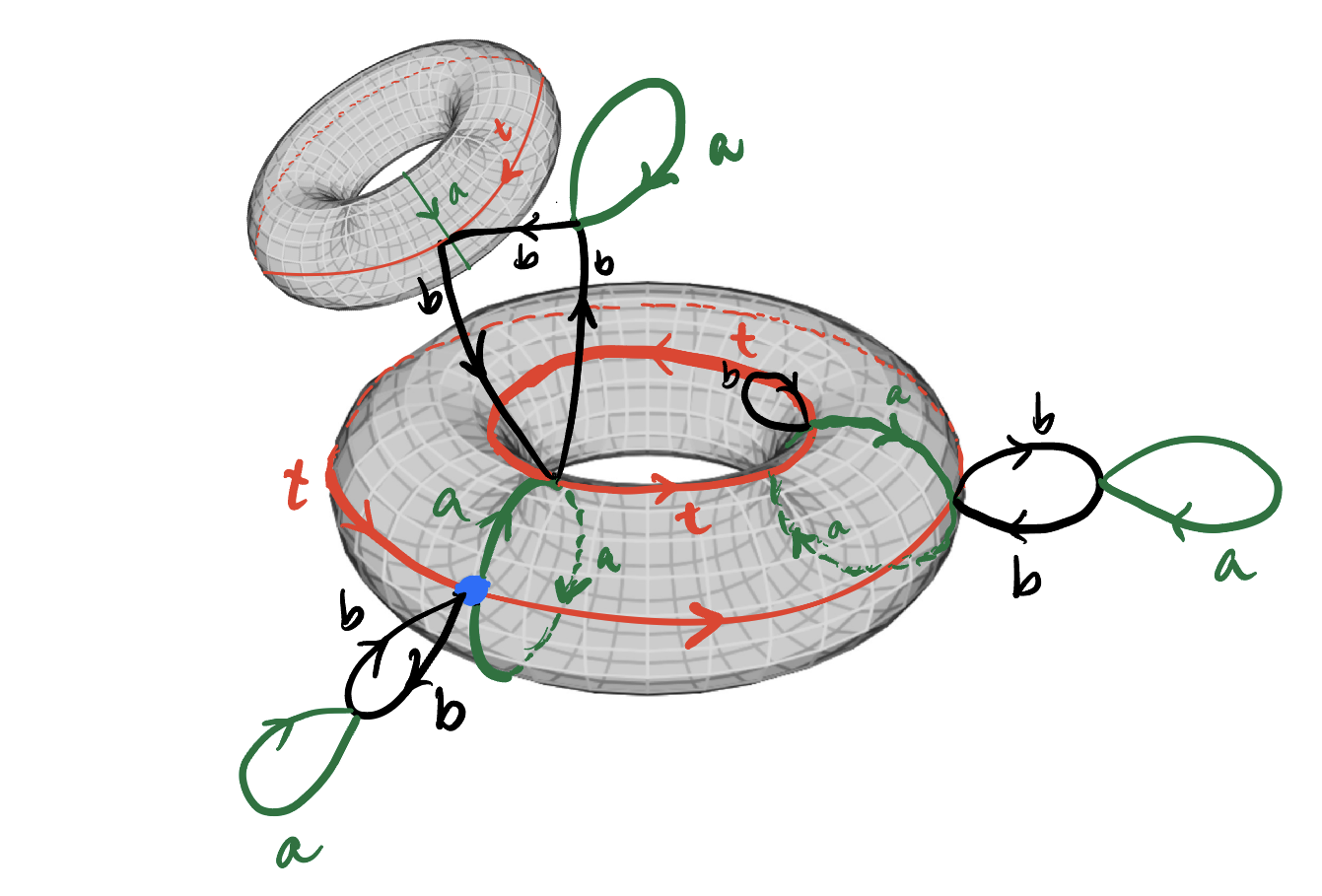}
\caption{Finite sheeted pre-cover $\bar X$}
\label{FSP2}
\end{figure}
\begin{figure}[ht!]
\centering
\includegraphics[width=0.6\textwidth]{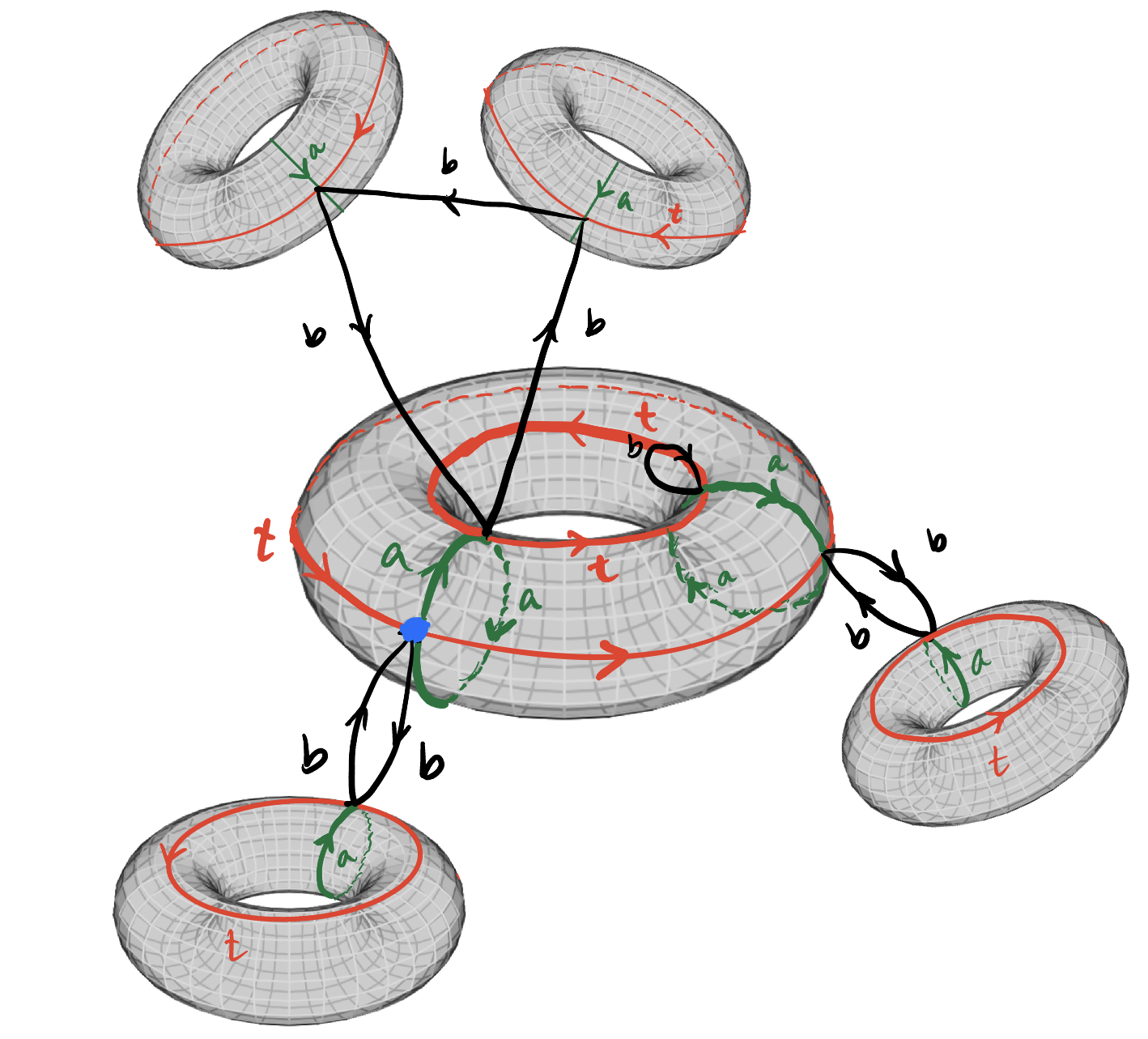}
\caption{Finite cover $\hat X$}
\label{cover2}
\end{figure}\newpage
\begin{rk} Theorem \ref{th2} is also true   when $L$ is abelian, therefore, free abelian. \end{rk} 
\begin{proof} We take a basis  $a_1,\ldots a_n$   of $L$ such that $H$ has a basis $a_1^{k_1},\ldots , a_r^{k_r}$. Let $$g=a_1^{m_1}\ldots  a_n^{m_n}.$$ 
If $m_i$ is not divisible by $k_i$ for some $i=1,\ldots ,r,$ then we take $K$ generated by $a_1^{k_1},\ldots , a_r^{k_r}, a_{r+1},\ldots ,a_n$.
If each $m_i$ is divisible by $k_i$ for $i=1,\ldots,r,$ then some of $m_{r+1},\ldots ,m_n$ is non-zero, because $g\not\in H.$ Suppose $m_n\neq 0.$ Then take $K$ generated by $a_1^{k_1},\ldots , a_r^{k_r}, a_{r+1},\ldots ,a_n^{m_n+1}.$
    
\end{proof}
\begin{rk}\label{rk2}
In the case when $H$ is abelian and  $L$ is non-abelian  a finite-index subgroup of $L$ cannot be fully residually $H$. In this case there exists $x\in L$ such that $g \notin H_1=\langle H,x\rangle =H*\langle x\rangle$. 
\end{rk}
\begin{proof} Take some $x\in L$ such that $[h,x]\neq 1$ for $h\in H$. Then for any $h\in H$, elements $h,x$ generate a free subgroup. Therefore
$H_1=\langle H,x\rangle =H*\langle x\rangle$. If $g\not\in H_1$, then we found $x$. If $g\in H_1$, then $g$    can be uniquely written as $$g=h_1x^{k_1}\ldots h_rx^{k_r},$$ where $h_1,\ldots, h_r$ are elements in $H$, all, except maybe $h_1$ non-trivial. Let $k$ be a positive number that is larger than all $|k_1|,\ldots ,|k_r|.$ Then $g\not\in H_2=\langle H,x^k\rangle$ and we can take $x^k$ instead of $x$.
\end{proof}

\section{Proof of Theorem \ref{th1}}
 \begin{df}\cite{L}
   Let $G$ be a finitely generated group and $H$ a finitely generated  subgroup of $G.$ For a complex affine algebraic group $\bf{G}$ and any representation $\rho_0\in Hom(G,\bf{G}),$ we have the closed affine subvariety 
    $$R_{\rho_0,H}(G,{\bf{G}})=\set{\rho\in Hom(G,{\bf G}): \rho_0(h)=\rho(h)\text{ for all } h\in H}$$
    The representation $\rho_0$ is said to \textit{strongly distinguish} $H$ in $G$ if there exist representations $\rho,\rho'\in R_{\rho_0,H}(G,\bf{G})$ such that $\rho(g)\neq\rho'(g)$ for all $g\in G-H.$
\end{df}

\noindent If $L$ is a closed surface group or a free group, then Theorem \ref{th1} follows from \cite[Theorem1.1]{L}.  Suppose $L$ is not a surface group and not an abelian group. Let $\bf G$ be a complex affine algebraic group. 
By the following  lemma, it is  sufficient to construct a faithful representation $\rho\in Hom(L,{\bf G})$ that strongly distinguishes $H$ in $L$. 

\begin{lm}\label{3.1}\cite[Lemma 3.1]{L} Let $G$ be a finitely generated group, ${\bf G}$ a complex algebraic group, and $H$ a finitely generated subgroup of $G.$ If $H$ is strong distinguished  by a representation $\rho\in Hom(G, {\bf G})$, then there exists a representation $\varrho: G\longrightarrow {\bf G}\times {\bf G}$ such that $\varrho(G)\cap\overline{\varrho(H)}=\varrho(H),$ where $\overline{\varrho(H)}$ is the Zariski closure of $\varrho(H)$ in ${\bf G}\times {\bf G}.$
\end{lm}

\noindent

\begin{prop}  Let $L$ be a limit group and $ H$ a non-abelian  finitely generated subgroup. There exist a finite-index subgroup $K\leq L$ and a faithful representation $\rho _{\omega}:K\rightarrow \bf G$ that  strongly distinguishes $H$ in $K$.
 \end{prop} 
 \begin{proof} 
 
By Theorem \ref{th2}, there exists a finite-index 
subgroup $K$ of $L$ such that $K$ is fully residually $H$. Let $\rho$ be a faithful representation of $H$ in $\bf G$. We order balls $B_t$ of radius $t$ in the Cayley graph of $K$ and finite sets $S_t=B_t\cap (K-H)$. Since we have a discriminating family of $H$-homomorphisms from $K$ to $H,$ we can construct for any $t\in\mathbb N$ representations $\rho _t$  and $\rho '_t$ in $Hom (K, \bf G)$ that coincide on $H$, distinguish all elements in  $S_t,$ and map $B_t$ monomorphically. Selecting a non-principal ultrafilter $\omega\in \bf N$, we have two associated ultraproduct representations $\rho _{\omega}, {\rho'}_{\omega}:K\rightarrow \bf G$ (see \cite[ Proof of Lemma 3.2]{L}).  These representations are faithful because each $B_t$ is mapped monomorphically on a co-finite set of $j\in\mathbb N$ and for any $g\in K-H$, $\rho_{\omega}(g)\neq {\rho'}_{\omega}(g).$
 \end{proof}

\noindent Let us prove the first statement of Theorem \ref{th1}.  The proof of \cite[Theorem 1.1]{L}
shows that it is sufficient to have a representation of $K$ that strongly distinguishes $H$. Indeed, like in \cite[Corollary 3.3]{L}, we can construct a representation
$\Phi : K\rightarrow GL(2,{\mathbb C})\times GL(2,{\mathbb C})$ such that $\Phi (g)\in Diag(GL(2,{\mathbb C}))$ if and only if $g\in H$. Setting $d_H=[G:K]$, we have the induced representation $${Ind_K}^G(\Phi ):G\rightarrow 
GL(2d_H,{\mathbb C})\times GL(2d_H,{\mathbb C}).$$  
Recall, that when $\Phi$ is represented by the action on the vector space $V$ and $G=\cup _{i=0} ^t g_iK$, then the induced representation 
acts on the disjoint union $\sqcup_{i=0} ^t g_iV$ as follows
$$g\Sigma g_iv_i=\Sigma g_{j(i)} \Phi (k_i)v_i,$$ where
$gg_i=g_{j(i)}k_i, $ for $k _i\in K.$ Taking $\rho ={Ind_K}^G(\Phi )$, it follows from the construction of $\rho $ and definition of induction that $\rho (g)\in\overline {( \rho (H))}$ if and only if $g\in H$. If we set $\rho=\rho_H,$  then  Theorem \ref{th1} is proved. \\

\section{ Proof of Corollary \ref{co1}}

Given a complex algebraic group $\textbf{G} < GL(n,\C)$, there exist polynomials  $P_1,\ldots,P_r\in\C\bkt{X_{i,j}}$ such that
 $$\textbf{G}=\textbf{G}\prn{\C}= V\prn{P_1,\ldots,P_r}=\set{X\in\C^{n^2}\mid P_k(X)=0, k=1,\ldots, r}$$
We refer to the polynomials $P_1,\ldots,P_r$ as \textit{defining polynomials}  for $\textbf{G}$. We will say that \textbf{G} is $K$–defined for a subfield $K\subset\C$ if there exists defining polynomials $P_1,\ldots,P_r\in K\bkt{X_{i,j}}$ for $\textbf{G}$.
For a complex affine algebraic subgroup $\textbf{H} < \textbf{G} < GL(n,\C)$, we will pick the defining polynomials for $\textbf{H}$ to contain a defining set for $\textbf{G}$ as a subset. Specifically, we have polynomials
$P_1,...,P_{r_{\textbf{G}}},P_{r_{\textbf{G}}+1},...,P_{r_{\textbf{H}}}$  such that 
 \begin{equation}\label{eq:1}
     \textbf{G}= V\prn{P_1,\ldots,P_{r_{\textbf{G}}}} \text{ and } \textbf{H}= V\prn{P_1,\ldots,P_{r_{\textbf{H}}}}
 \end{equation}

\noindent If $\textbf{G}$ is defined over a number field $K$ with associated ring of integers $\mathcal{O}_K$, we can find polynomials $P_1,\ldots,P_r\in\mathcal{O}_K\bkt{X_{i,j}}$  as a defining set by clearing denominators. For instance, in the case when $K = \Q$ and $\mathcal{O}_K= \Z$, these are multivariable integer polynomials.\vspace{2mm}

\noindent For a fixed finite set $X=\set{x_1,\ldots, x_t}$ with associated free group $F(X)$ and any group $G,$ the set of homomorphisms from $F(X)$ to $G,$ denoted by $Hom\prn{F\prn{X},G},$ can be identified with $G^t=G_1\times\ldots\times G_t.$ For any point $\prn{g_1,\ldots,g_t}\in G^t,$ we have an associated homomorphism $\varphi_{\prn{g_1,\ldots,g_t}}:F\prn{X}\longrightarrow G$ given by $\varphi_{\prn{g_1,\ldots,g_t}}\prn{x_i}=g_i.$ For any word $w\in F(X),$ we have a function Eval$_w: Hom(F(X), G)\longrightarrow G $ defined by  Eval$_w(\varphi_{\prn{g_1,\ldots,g_t}})(w)=w(g_1,\ldots,g_t).$ For a finitely presented group $\Gamma,$ we fix a finite presentation $\gen{\gamma_1,\ldots, \gamma_{t}\mid r_1,\ldots, r_{t'} },$ where $X=\set{\gamma_1,\ldots, \gamma_{t}}$ generates $\Gamma$ as a monoid and $\set{r_1,\ldots, r_{t'}}$ is a finite set of relations. If $\textbf{G}$ is a complex affine algebraic subgroup of $Gl_n(n,\C),$ the set $Hom(\Gamma,\textbf{G})$ of homomorphisms $\rho: \Gamma\longrightarrow \textbf{G}$ can be identified with an affine subvariety of $G^t.$ Specifically,
\begin{equation}
    Hom(\Gamma,\textbf{G})=\set{\prn{g_1,\ldots, g_{t}}\in \textbf{G}^t\mid r_j\prn{g_1,\ldots, g_{t}}=I_n \text{ for all } j}
\end{equation}

\noindent If $\Gamma$ is finitely generated, $Hom(\Gamma,\textbf{G})$ is an affine algebraic variety by the Hilbert Basis Theorem.

\vspace{2mm}\noindent The set $Hom(\Gamma,\textbf{G})$ also has a topology induced by the analytic topology on $G^t.$ There is a Zariski open
subset of $Hom(\Gamma,\textbf{G})$ that is smooth in the this topology called the smooth locus, and the functions
Eval$_w: Hom(\Gamma, \textbf{G})\longrightarrow \textbf{G}$ are analytic on the smooth locus. For any subset $S\in G$ and representation $\rho\in Hom(\Gamma, \textbf{G})$, $\overline{\rho(S)}$ will denote the Zariski closure of $\rho(S)$ in $\textbf{G}$.

 \begin{lm}\label{lm11} (\cite[Lemma 5.1]{L})
    Let ${\bf G}\leq GL\prn{n,\C}$ be a $\overline{\Q}$-algebraic group, $L\leq {\bf G}$ be a finitely generated subgroup, and $\textbf{A}\leq {\bf G}$ be a  $\overline{\Q}$-algebraic subgroup. Then, $H=L\cap\textbf{A}$ is closed in the profinite topology. 

 \end{lm}

\begin{proof}
    Given $g\in L-H$, we need a homomorphism $\varphi:L\longrightarrow Q$ such that $|Q|<\infty$ and $\varphi\prn{g}\notin\varphi\prn{H}.$
We first select polynomials $P_1,...,P_{r_{\textbf{G}}},...,P_{r_{\textbf{A}}}\in\C\bkt{X_{i,j}} $ satisfying (\ref{eq:1}). Since $\textbf{G}$ and $\textbf{A}$ are $\overline{\Q}$-defined, we can select $P_j\in\mathcal{O}_{K_0}\bkt{X_{i,j}}$ for some number field $K_0/\Q.$ We fix a finite set $\set{l_1,\ldots, l_{r_L}}$ that generates $L$  as a monoid. In order to distinguish between elements of $L$ as an abstract group and the explicit elements in ${\bf G},$ we set $l=M_l\in {\bf G}$ for each $l\in L.$  In particular, we have a representation given by $\rho_0:L\longrightarrow{\bf G}$ given by $\rho_0(l_{t})=M_{l_{t}}$. We set $K_L$ to be the field generated over $K_0$ by the set of matrix entries $\set{\prn{M_t}_{i,j}}_{t,i,j}$. It is straightforward to see that $K_L$  is independent of the choice of the generating set for $L.$  Since $L$ is finitely generated, the field $K_L$ has finite transcendence degree over $\Q$ and so $K_L$ is isomorphic to a field of the form $K(T)$  where $K/\Q$ is a number field and  $T=\set{T_1,\ldots, T_d}$ is a transcendental basis (See \cite{L}). For each, $M_{l_t},$ we have $(M_{l_t})_{i,j}=F_{i,j,t}(T)\in K_L$. In particular, we can view the $(i,j)$–entry of the matrix $M_{l_t}$ as a rational function in $d$ variables with coefficients in some number field $K$. Taking the ring
generated over $\mathcal{O}_{K_0}$ by the set $\set{\prn{M_{l_t}}_{i,j}}_{t,i,j}$, $R_L$ is obtained from $\mathcal{O}_{K_0}\bkt{T_1,\ldots, T_d}$ by inverting
a finite number of integers and polynomials. Any ring homomorphism  $R_L\longrightarrow R$ induces a group homomorphism $GL(n,R_L)\longrightarrow GL(n,R)$, and since $L\leq GL(n,R_L)$, we obtain $L\longrightarrow GL(n,R)$ . If $g\in L- H$ then there exists $r_{\mathbf{G}}<j_g\leq r_{\mathbf{A}} $ such that $Q_g=P_{j_{g}}\prn{\prn{M_l}_{1,1},\ldots,\prn{M_l}_{n,n}}\neq 0$. Using Lemma 2.1 in \cite{KM}, we have a ring homomorphism $\psi_R:R_L\longrightarrow R$ with $|R|<\infty$ such that $\psi_R(Q_g)\neq 0$. Setting, $\rho_{R}:GL(n, R_L)\longrightarrow GL(n,R)$ we assert that $\rho_{R}(g)\notin\rho_{R}(H)$. To see this, set $\overline{   M}_{\eta}=\rho_R(\eta)$ for each $\eta\in L$, and note that $\psi_R(P_j((M_{\eta})_{1,1},\ldots,M_{\eta})_{n,n}) )=P_j((\overline{M}_{\eta})_{1,1},\ldots,(\overline{M}_{\eta})_{n,n} )$ . For each $h\in H$, we know that $P_{j_{l}}\prn{(M_h)_{i,j}}=0$ and so $P_j((\overline{M}_{\eta})_{1,1},\ldots,(\overline{M}_{\eta})_{n,n})= 0$ . However, by selection of $\psi_R$, we know that $\psi_R(Q_g)\neq 0$ and so $\rho_{R}(g)\notin\rho_{R}(H)$ . 
\end{proof}

 \noindent Theorem \ref{th1}  and Lemma \ref{lm11} imply Corollary \ref{co1}. 

\begin{proof} 
    Since $H\leq L$ is finitely generated, by Theorem \ref{th1}, there is a faithful representation $$\rho_H:L\longrightarrow GL\prn{n,\C}$$ such that $\overline {\rho _H(H)} \cap\rho_H(L)=\rho _H(H)$. We can construct the representation in Theorem \ref{th1} so that $\textbf{G}=\overline{\rho_H(L)}$ and $\textbf{A}=\overline {\rho _H(H)}$ are both $\overline{\Q}$-defined. So, by Lemma \ref{lm11}, we can separate $H$ in $L.$ Next, we quantify the separability of $H$ in $L.$ Toward that end, we need to bound the order of the ring $R$ in the proof of Lemma \ref{lm11} in terms of the word length of the element $g.$ Lemma 2.1 from \cite{KM} bounds the size of $R$ in terms of the coefficient size and degree of the polynomial $Q_g.$ It follows from a discussion on pp 412-413 of \cite{KM} that the coefficients and degree can be bounded in terms of the word length of $g,$ and that the coefficients and degrees of the polynomials $P_j.$ Because the $P_j$ are independent of the word $g,$ there exists a constant $N_0$ such that $|R|\leq \norm{g}^{N_0}.$ By construction, the group $Q$ we seek is a subgroup of $GL(n,R).$ Thus, $|Q|\leq |R|^{n^2}\leq\norm{g}^{N_0n^2}.$ Taking $N=N_0n^2$ completes the proof.



\end{proof}
\section{The Hanna Neumann conjecture for hyperbolic limit groups}\label{HN}
Y. Antolin and A. Jaikin-Zapirain proved in \cite{AJ} the geometric Hanna Neumann conjecture for surface groups and formulated 
the Geometric Hanna Neumann conjecture for limit groups \cite[Conjecture 1]{AJ}as follows. Let G be a limit group. Then for
every two finitely generated subgroups $U$ and $W$ of $G$
$$ \Sigma _{x\in U\backslash G/W}  \bar\chi (U\cap xWx^{-1})\leq \bar \chi (U)\bar\chi (W)$$

Here for a virtually FL-group $\Gamma$ we define its Euler characteristic as
$$\chi (\Gamma )=\frac{1}{[\Gamma :\Gamma _0]}\Sigma _{i=0}^{\infty} (-1)^i dim_{\mathbb Q}H_i(\Gamma _0,\mathbb Q), $$ where $\Gamma _0$ is an FL-subgroup of $\Gamma$ of finite index. And $\bar\chi (\Gamma)=\max\{0,-\chi (\Gamma )\}.$ Observe that for a non-trivial finitely generated free group $\Gamma$, $\bar\chi (\Gamma )=d(\Gamma )-1$, where $d(\Gamma )$ is the number of generators, for a surface group $\Gamma$ we have 
$\bar\chi (\Gamma)=d(\Gamma )-2$. By a surface group we mean the fundamental group of a compact closed surface of negative
Euler characteristic. Notice that by \cite{AJ} limit groups are FL-groups. Notice also that for hyperbolic limit groups $dim_{\mathbb Q}H_i(\Gamma _0,\mathbb Q)=0$ for $i>2$.

In this section we will prove the conjecture for hyperbolic limit groups.

The notion of $L^2$-independence was introduced in \cite{AJ}.  The group $G$ is $L^2$-Hall, if for every finitely generated subgroup $H$ of $G$, there
exists a subgroup $K$ of $G $ of finite index containing $H$ such that $H$ is $L^2$-independent in $K$. Let $G$ be a hyperbolic limit group. By \cite[Theorem 1.3]{AJ}, if   $G$ satisfies the $L^2$-Hall property,
then the geometric Hanna Neumann conjecture holds for $G.$

As explained in \cite[Lemma
4.1]{AJ} and the comment after the lemma, since the limit groups satisfy the strong
Atiyah conjecture, if $G$ is a limit group and $H\leq K$ subgroups in $G$, then $H$ is
$L^2$-independent in $K$ if the correstriction map
$$cor : H_1(H; \mathcal D_{\mathbb Q[G]})\rightarrow  H_1(K; \mathcal D_{\mathbb Q[G]})$$ is injective. Here
$\mathcal D_{\mathbb Q[G]}$ denote the Linnell division ring.
\begin{lm} \label{lmL} Let $G$ be a limit group and $H\leq K$ subgroups of $G$. Assume that there
an abelian subgroup $B$ of $G$ such that $K = \langle H,B\rangle =H*_{A}B, $ 
where $A = H\cap B$.
Then the correstriction map $cor : H_1(H; \mathcal D_{\mathbb Q[G]})\rightarrow  H_1(K; \mathcal D_{\mathbb Q[G]})$ is injective.

\end{lm} 
\begin{proof} By \cite[Theorem2(2)]{Ch}, we obtain the exact sequence
$$H_1(A; \mathcal D_{\mathbb Q[G]})\rightarrow ^{(cor, -cor )}  H_1(H; \mathcal D_{\mathbb Q[G]})\oplus H_1(B; \mathcal D_{\mathbb Q[G]})\rightarrow^{(cor, cor )}  H_1(K; \mathcal D_{\mathbb Q[G]}).$$

Since $A$ is abelian, $H_1(A;D_{Q[G]})=0$. Indeed,  the division ring generated by $Q[A]$ inside $D_{Q[G]}$  is isomorphic to  the field of fractions $R$ of $Q[A]$, and so $D_{Q[G]}$ is also an $R$-vector space. Thus, $D_{Q[G]}$ is flat as a $Q[A]$-module. In particular, $H_1(A;D_{Q[G]})=0$.

So the correstriction map $$H_1(H; \mathcal D_{\mathbb Q[G]})\rightarrow  H_1(K; \mathcal D_{\mathbb Q[G]})$$ is injective.

\end{proof}
\begin{cy} A limit group is $L^2$-Hall.
\end{cy}
\begin{proof} Let $G$ be an ICE- group and $H$ a finitely generated subgroup of $G$. Then by Theorem \ref{th3} there  exists a finite chain of groups $H=K_0<\ldots <K_n=K$ with $K$ of finite index in $G$, where $K_{i+1}$ is either $K_i*\mathbb Z$ or $K_{i+1}$ is an extension of a centralizer of $K_i$. By Lemma \ref{lmL}, the correstriction maps 
$$H_1(K_i; \mathcal D_{\mathbb Q[G]})\rightarrow  H_1(K_{i+1}; \mathcal D_{\mathbb Q[G]})$$ are injective. Hence $H$ is $L^2$-independent in $K$. 
Now, let $H<L<G$, then $H$ will be $L^2$-independent in $L\cap K$ because
the composition of correstrictions maps is correstriction. Thus $L$ is $L^2$-Hall. 
\end{proof}
Therefore we obtain the following theorem.
\begin{theorem} \label{thHN}The geometric Hanna Neumann conjecture is true for hyperbolic limit groups.
\end{theorem}

\vspace{.2cm}
{\large Acknowledgements}

We thank  A. Vdovina and H. Wilton for very useful discussions.  We  thank A. Jaikin-Zapirain for explaining how the Hanna Neumann conjecture follows from Theorem \ref{th3}.


\begin{thebibliography}{10}





\bibitem {AJ} Y.Antolin, A. Jaikin-Zapirain, The Hanna Neumann conjecture for surface groups.
Compos. Math. 158 (2022), no. 9, 1850-1877.
\bibitem{BR1} K. Bou-Rabee, Quantifying residual finiteness. J. Algebra 323, 729-737 (2010)
\bibitem{BR2} K. Bou-Rabee, M.F Hagen, P. Patel, Residual finiteness growths of virtually special groups. Math.
Z. 279, 297-310 (2015)
\bibitem{BR3} K. Bou-Rabee,  T. Kaletha,  Quantifying residual finiteness of arithmetic groups. Compos.Math. 148,
907-920 (2012)
\bibitem{BR4} K. Bou-Rabee,  D.B. McReynolds, Asymptotic growth and least common multiples in groups. Bull.
Lond. Math. Soc. 43, 1059-1068 (2011)
\bibitem{KM} K. Bou-Rabee, D.B. McReynolds, Extremal behavior of divisibility functions on linear groups.Geom.
Dedicata 175, 407-415 (2015).
\bibitem{B} N.V. Buskin, Efficient separability in free groups. Sibirsk. Mat. Zh. 50, 765-771 (2009)
\bibitem{Ch} I.M. Chiswell,  Exact sequences associated with a graph of groups. J. Pure Appl. Algebra 8
(1976), no. 1, 63-74.
\bibitem{HP}  M.F. Hagen,  P. Patel,  Quantifying separability in virtually special groups. Pacific J. Math. 284, 103-120 (2016)
\bibitem{L} L. Louder, D. B. McReynolds,
P.Patel, Zariski closures and subgroup separability, Sel. Math. New Ser. (2017) 23, 2019-2027. 
\bibitem{Lyndon} R. Lyndon, Groups with parametric exponents, Trans. Amer. Math. Soc., 95, 518-533, 1960.
\bibitem{K1} M. Kassabov, F. Matucci, Bounding the residual finiteness of free groups. Proc. Am.Math. Soc. 139,
2281-2286 (2011)
\bibitem{K2} G. Kozma,  A. Thom,  Divisibility and laws in finite simple groups. Math. Ann. 364, 79-95 (2016)
\bibitem{K3}  O. Kharlampovich,  A.  Myasnikov, M.  Sapir,  Algorithmically complex residually finite groups, Bulletin of Mathematical Sciences 7 (2017), no.
2, 309-352.
\bibitem{KM2} O.Kharlampovich, A. Myasnikov, Irreducible affine varieties over a free group II. Systems in triangular quasi-quadratic form and description of residually free groups, J. Algebra, 200, 1998, 517-570.
\bibitem{KMS}  O. Kharlampovich, A. Myasnikov, R. Sklinos, Frasse Limits of Limit Groups,
Journal of Algebra, 545 (2020) 300-323.

\bibitem{P1} P. Patel, On a Theorem of Peter Scott. Proc. Am. Math. Soc. 142, 2891-2906 (2014)
\bibitem{P2} P. Patel,  On the residual finiteness growths of particular hyperbolic manifold groups. Geom. Dedicata
186(1), 87-103 (2016)
\bibitem{P3}  M. Pengitore, Effective conjugacy separability of finitely generated nilpotent groups. http://arxiv.org/
abs/1502.05445 (2015)
\bibitem{R} I. Rivin,  Geodesics with one self-intersection, and other stories. Adv. Math. 231, 2391-2412 (2012)
\bibitem{S} B. Solie, 
Quantitative residual properties of $\Gamma$- imit groups. (English summary)
Internat. J. Algebra Comput. 24 (2014), no. 2,207-231.
\bibitem {Wilton} H. Wilton, Hall's Theorem for Limit Groups, Geometric and Functional Analysis, vol. 18, no. 1, Apr. 2008, pp. 271-303. 

\end{thebibliography}
\end{document}